\documentclass[10pt,leqno]{article}  
\usepackage{amsmath,amstext, amsthm,amsfonts,amssymb,amsthm}  
 \numberwithin{equation}{section} 
\def\C{\mathbb{ C}}

\usepackage{epsfig}
\newtheorem{thm}{Theorem}[section]
\newtheorem{lem}[thm]{Lemma}

\newtheorem{cor}[thm]{Corollary}

\newcommand{\be}{\begin{equation}}
\newcommand{\ee}{\end{equation}}

\newcommand{\ba}{\begin{array}}
\newcommand{\ea}{\end{array}}

\newcommand{\bg}{\begin{gathered}}
\newcommand{\eg}{\end{gathered}}

\newcommand{\al}{\alpha}

\newcommand{\f}{\phi}

\newcommand{\bea}{\begin{eqnarray}}
\newcommand{\eea}{\end{eqnarray}}

\newcommand{\Sum}{\sum_{n=0}^\infty}

\begin{document}

\title{Integrals and Series Representations of $q$-Polynomials and 
Functions: Part II\\
 Schur Polynomials  and the 
Rogers-Ramanujan Identities}

\author{ Mourad E. H.  Ismail 
 \thanks{Research partially supported  by the DSFP of King Saud 
 University and by  the National Plan for Science, Technology and 
 innovation (MAARIFAH), King Abdelaziz City for Science and 
 Technology, Kingdom of Saudi Arabia, Award number 
14-MAT623-02.} \\
 \and Ruiming Zhang \thanks{Corresponding author, research partially supported by National 
 Science Foundation of China, grant No. 11371294.}}
\maketitle

\begin{abstract}
 We  give several expansion and identities involving the Ramanujan 
 function $A_q$  and the Stieltjes--Wigert polynomials. Special values 
 of our idenitities give $m$-versions of some of the items on the 
 Slater list of Rogers-Ramanujan type identities.  We also study some 
 bilateral extensions of certain transformations in the theory of 
 basic hypergeometric functions. 
 \end{abstract}
 {\bf Mathematics Subject Classification MSC 2010}: Primary 11P84, 33D45 Secondary  05A17. 
 \noindent
 \\
 

{\bf Filename}:  IsmZhaRRII1

\section{Introduction}  
In part I of this series of papers we derived several integral 
representations for many $q$-functions and polynomials including the 
the $q$-exponential functions $e_q, E_q$, and our ${\mathcal E}_q$, 
\cite{Ism:Zha}. We also derived several series identities and 
transformations. The present work is part II where we continue our 
studies and establish quite a few series  identities and 
transformation formulas. Some of our formulas give new identities for 
the Schur polynomials introduced by I. Schur in \cite{Sch}.  We also 
treat a generalization of the Schur polynomials we introduced in 
\cite{Ism:Zha3}.

We recall the  Rogers--Ramanujan identities  are 
\bea
\bg
\Sum \frac{q^{n^2}}{(q;q)_n} = \frac{1}{(q, q^4;q^5)_\infty}, \qquad 
\Sum \frac{q^{n^2+n}}{(q;q)_n} = \frac{1}{(q^2, q^3;q^5)_\infty},
 \eg
 \label{eqRR}
 \eea
 where the notation for the $q$-shifted factorials is the standard 
 notation followed in \cite{Gas:Rah},  \cite{And:Ask:Roy}, or 
 \cite{Ismbook}.  References for the Rogers-Ramanujan identities, 
 their origins and many of 
 their applications are in \cite{And2}, \cite{And}, and \cite{And:Ask:Roy}.   
 Garrett, Ismail, and Stanton 
 \cite{Gar:Ism:Sta}  generalized the Rogers--Ramanujan to   
  \bea
 \Sum \frac{q^{n^2+mn}}{(q;q)_n} =\frac{(-1)^m q^{-\binom{m}{2}} a_m(q)}{(q,q^4;q^5)_\infty}
- \frac{(-1)^{m} q^{-\binom{m}{2}} b_m(q)}{(q^2,q^3;q^5)_\infty},
\label{eqmform}
 \eea
 where $a_m(q)$ and $b_m(q)$ are defined by 
 \begin{equation}
\label{eq13.6.1}
\begin{gathered}
a_m(q)=\sum_{j\ge 0} q^{j^2+j}\bmatrix m-j-2 \\ j\endbmatrix_q, 
\qquad 
b_m(q) =\sum_{j\ge 0} q^{j^2}\bmatrix m-j-1 \\ j\endbmatrix_q.
\end{gathered}
\end{equation}
The Garrett--Ismail--Stanton result became known as 
the $m$-version of the Rogers-Ramanujan identities. 

The polynomials $a_m(q)$ and $b_m(q)$ were considered by Schur in conjunction with his proof of the  
Rogers--Ramanujan identities, see
\cite{And2}  and \cite{Gar:Ism:Sta}   for details. We shall refer to 
$a_m(q)$ and $b_m(q)$ as the Schur polynomials. The closed form 
expressions for $a_m$ and $b_m$ in  \eqref{eq13.6.1}  were given 
by  Andrews in \cite{And4}, where he also  
gave a polynomial generalization of the Rogers--Ramanujan 
identities.  We must note that $a_{m+1}(q)$ and $b_{m+1}(q)$ are the partial numberators and denominators of the Ramanujan continued fraction, 
   
  In this work we will use the following  confluent limit of the $q$-Gauus 
  sum
  \begin{eqnarray}
  \label{eqqCGauss}
{}_{1}\phi_{1}\left(a;c;q,c/a\right) & = & \frac{(c/a;q)_{\infty}}
{(c;q)_{\infty}}
\end{eqnarray}
  We will also use the integral representations,  
  \bea
  \bg
\left(bq^{\alpha};q\right)_{\infty}q^{\alpha^{2}/2}
{}_{1}\phi_{1}\left(a;bq^{\alpha};q,zq^{\alpha+1/2}\right) \\
=\frac{1}{\sqrt{\pi\log q^{-2}}}\int_{-\infty}^{\infty}
\frac{\left(-aze^{ix};q\right)_{\infty}\exp\left(\frac{x^{2}}
{\log q^{2}}+i\alpha x\right)dx}
{\left(-bq^{-1/2}e^{ix},-ze^{ix};q\right)_{\infty}},
\eg
\label{eq:1}
\eea 
\begin{equation}
q^{\alpha^{2}/2}A_{q}\left(q^{\alpha}z\right)=\frac{1}{\sqrt{2\pi}}
\int_{-\infty}^{\infty}\frac{\left(zq^{1/2}e^{ix};q\right)_{\infty}
\exp\left(\frac{x^{2}}{\log q^{2}}+i\alpha x\right)}{\sqrt{\log q^{-1}}}dx,
\label{eq:2}
\end{equation}
\begin{equation}
q^{\alpha^{2}/2}\left(-zq^{\alpha+1/2};q\right)_{\infty}=\int_{-\infty}^{\infty}\frac{\exp\left(x^{2}/\log q^{2}+i\alpha x\right)dx}{\left(ze^{ix};q\right)_{\infty}\sqrt{\pi\log q^{-2}}},\label{eq:qe1}
\end{equation}
and
\begin{equation}
q^{\alpha^{2}/2}S_{n}\left(xq^{\alpha-1/2};q\right)=\int_{-\infty}^{\infty}\frac{\left(xe^{iy};q\right)_{n}}{\left(q;q\right)_{n}}\frac{\exp\left(y^{2}/\log q^{2}+i\alpha y\right)dy}{\sqrt{\pi\log q^{-2}}}.
\label{eq:sw1}
\end{equation}
  which we proved in our forthcoming paper \cite{Ism:Zha4}.  
  
  In Section 2  we prove  the following theorem and discuss some of its implications. 
  \begin{thm}\label{thm1}
  We have the indentites
  \bea
 (-q;q)_{\infty} & =&\frac{1}{(q,q^{4};q^{5})_{\infty}}
 \sum_{n=0}^{\infty}\frac{q^{2n}a_{2n}(q)}{(q^{2};q^{2})_{n}}
 -\frac{1}{(q^{2},q^{3};q^{5})_{\infty}}\sum_{n=0}^{\infty}
 \frac{q^{2n}b_{2n}(q)}{(q^{2};q^{2})_{n}}. 
\label{eq:6}\\
(-q^{2};q)_{\infty}&=&\frac{1}{(q^{2},q^{3};q^{5})_{\infty}}
\sum_{n=0}^{\infty}\frac{q^{2n}b_{2n+1}(q)}{(q^{2};q^{2})_{n}}
-\frac{1}{(q,q^{4};q^{5})_{\infty}}\sum_{n=0}^{\infty}
\frac{q^{2n}a_{2n+1}(q)}{(q^{2};q^{2})_{n}}.\label{eq1.8}
\eea
  \end{thm}
  This theorem will follow from Theorem \ref{thm2} which we now state. 
  \begin{thm}\label{thm2} The following identities hold for all $a, z
   \in \C$ but $z \ne q^n, n=0, -1, -2, \cdots$,
  \bea
  \left(z;q\right)_{\infty}{}_{1}\phi_{1}\left(a;z;q,-z\right)  &=&\sum_{n=0}^{\infty}\frac{q^{2n^{2}-n}z^{2n}}{(q^{2};q^{2})_{n}}A_{q}\left(q^{2n-1}az\right),
\label{eq1.9}\\
\left(b;q\right)_{\infty}{}_{1}\phi_{1}\left(a;b;q,-b\right)&=&
\sum_{n=0}^{\infty}\frac{q^{n^{2}-n}\left(-ab\right)^{n}}
{\left(q;q\right)_{n}}A_{q^{2}}\left(-q^{2n-1}b^{2}\right),
\label{eq:exponential 18}\\
\sum_{n=0}^{\infty}\frac{q^{2n^{2}-n}b^{2n}}{\left(q^{2};q^{2}\right)_{n}}
A_{q}\left(q^{2n-1}ab\right)&=&\sum_{n=0}^{\infty}
\frac{q^{n^{2}-n}\left(-ab\right)^{n}}{\left(q;q\right)_{n}}
A_{q^{2}}\left(-q^{2n-1}b^{2}\right).\label{eq:exponential 21}
  \eea
  \end{thm}
  Formula \eqref{eq1.9} is (7.32) in our paper \cite{Ism:Zha3} while 
  \eqref{eq:exponential 18} is the result of writing $A_q(q^{2n-1}az)$ 
  as a sum then interchange the two sums, which is then stated formally as \eqref{eq:exponential 21}.  
  \begin{cor}\label{cor1}
  We have 
  \begin{equation}
(-z;q)_{\infty}=\sum_{n=0}^{\infty}\frac{q^{2n^{2}-n}z^{2n}}{(q^{2};q^{2})_{n}}A_{q}\left(-q^{2n-1}z\right),\label{eq1.12}
\end{equation}
  \end{cor}
  This is just the case $a=1$ of \eqref{eq1.9}.  Recall the definition of the Jackson  modified $q$-Bessel functions, using the  
  notation from, \cite{Ism82},
\begin{eqnarray}
I_\nu^{(1)}(z;q) &=&  \frac{(q^{\nu+1};q)_\infty} {(q;q)_\infty} 
\Sum \frac{   (z/2)^{\nu+2n}}{(q, q^{\nu+1};q)_n}, \quad    |z| < 2, 
\label{eqJnu1}\\
I_\nu^{(2)}(z;q) &=& \frac{(q^{\nu+1};q)_\infty} {(q;q)_\infty}  
\Sum \frac{  q^{n(n+\nu)}}{(q, q^{\nu+1};q)_n} (z/2)^{\nu+2n},
\label{eqJnu2}\\
I_\nu^{(3)}(z;q) &=& \frac{(q^{\nu+1};q)_\infty} {(q;q)_\infty}  
\Sum \frac{  q^{\binom{n}{2}}}{(q, q^{\nu+1};q)_n} (z/2)^{\nu+2n}.  
\label{eqJnu3}
\end{eqnarray}

    In our work \cite{Ism:Zha4} we 
  proved Theorem \ref{thm3} below. 
 \begin{thm}\label{thm3}
We have the inverse pair
\bea
J_{\nu}^{(2)}\left(2z;q\right)&=&\frac{z^{\nu}}{\left(q;q\right)_{\infty}}
\sum_{k=0}^{\infty}\frac{\left(-q^{\nu}\right)^{k}}{\left(q;q\right)_{k}}
q^{\binom{k+1}{2}}A_{q}\left(q^{\nu+k}z^{2}\right), 
\label{eq:series and identities bessel 5 a}\\
\frac{z^{\nu}A_{q}\left(q^{\nu}z^{2}\right)}{\left(q;q\right)_{\infty}} &=&
\sum_{k=0}^{\infty}\frac{q^{k^{2}}}{\left(q;q\right)_{k}}
\left(\frac{q^{\nu}}{z}\right)^{k}J_{k+\nu}^{(2)}\left(2z;q\right).
\label{eq:series and identities bessel 5 b}
\eea
\end{thm}
 In the same work we also proved that 
 \bea 
\frac{A_{q}\left(z\right)}{\left(zq^{2};q^{2}\right)_{\infty}}
&=&\sum_{k=0}^{\infty}\frac{\left(-z\right)^{k}q^{k^{2}}}
{\left(q^2,zq^2;q^2\right)_{k}}.
\label{eq1.18}\\
S_{n}\left(x;q\right)&=&\frac{1}{\left(q;q\right)_{n}} 
 \sum_{k=0}^{\infty}\frac{\left(xq^{n}\right)^{k}}{\left(q;q\right)_{k}}
 q^{\binom{k+1}{2}}
  A_{q}\left(q^{k}x\right), \label{eq:series and identities sw 3}\\
A_{q}\left(ab\right) &=& \sum_{k=0}^{\infty}\frac{\left(b;q\right)_{k}}
{\left(q;q\right)_{k}}
 q^{\binom{k+1}{2}}a^{k}A_{q}\left(aq^{k}\right), 
\label{eq:series and identities airy 1}
\eea
see equations (7.6), (7.8), and (7.2), respectively,  in   \cite{Ism:Zha4}. 
The Stieltjes--Wigert polynomials $S_n$ in 
\eqref{eq:series and identities sw 3} are defined by 
\bea
S_n(x) = \sum_{k=0}^n \frac{q^{k^2}}{(q;q)_k(q;q)_{n-k}} (-z)^k.
\label{eqdefSWpol}
\eea
In Section 2 we also discuss some consequences of the results stated 
so far. This leads, among other 
things,  to new generating functions for the Schur polynomials.   
Section 3 contains a master identity for bilateral $q$-series and some 
noteworthy special and limiting cases of it. Section 4 has an extensive 
list of new identities involving the Ramanujan function and special cases 
of them lead to $m$-versions of formulas on the Slater list \cite{Sla}. It 
also contains several new results on the Stieltjes--Wigert polynomials.

   \section{Theorems \ref{thm1}-\ref{thm3}}
   It is obvious that \eqref{eqmform} is nothing but
\begin{equation}
A_{q}(-q^{m})=\sum_{n=0}^{\infty}\frac{q^{n^{2}+mn}}{(q;q)_{n}}=\frac{(-1)^{m}q^{-\binom{m}{2}}a_{m}(q)}{(q,q^{4};q^{5})_{\infty}}-\frac{(-1)^{m}q^{-\binom{m}{2}}b_{m}(q)}{(q^{2},q^{3};q^{5})_{\infty}}.\label{eq2.1}
\end{equation}
This formula will be used repeatedly in this work. 

  Corollary \ref{cor1} is the case $a= 1$ of  Theorem \ref{thm2}.
  
\begin{proof}[Proof of Theorem \ref{thm1}]
We take $z =q$ in Cor \ref{cor1} and apply \eqref{eqmform} to obtain 
\begin{eqnarray*}
(-q;q)_{\infty} & = & \sum_{n=0}^{\infty}
\frac{q^{2n^{2}+n}}{(q^{2};q^{2})_{n}}
A_{q}\left(-q^{2n}\right)\\
 & = & \sum_{n=0}^{\infty}\frac{q^{2n^{2}+n}}{(q^{2};q^{2})_{n}}
 \left(\frac{q^{-\binom{2n}{2}}a_{2n}(q)}{(q,q^{4};q^{5})_{\infty}}
 -\frac{q^{-\binom{2n}{2}}b_{2n}(q)}{(q^{2},q^{3};q^{5})_{\infty}}\right),
\end{eqnarray*}
 which simplifies to \eqref{eq:6}. Similarly 
 $z=q^{2}$ in \eqref{eq1.12} leads to \eqref{eq1.8}. 
\end{proof}

 Let $a=0$ in (\ref{eq1.9}) to get
\begin{equation}
\left(z;q\right)_{\infty}\sum_{n=0}^{\infty}\frac{q^{\binom{n}{2}}z^{n}}
{(q,z;q)_{n}}=\sum_{n=0}^{\infty}\frac{q^{2n^{2}-n}z^{2n}}
{(q^{2};q^{2})_{n}}.
\notag
\end{equation}
Obviously the above equation is 
\bea
A_{q^2}(-z^2/q) = \left(z;q\right)_{\infty}\sum_{n=0}^{\infty}
\frac{q^{\binom{n}{2}}z^{n}}{(q,z;q)_{n}},
\label{eq2.5} 
\eea
which relates the Ramanujan function to a modified $q$-Bessel 
function $I_\nu^{(3)}$.  The special case $z = q^{m+1/2}$ is of interest 
and gives 
   \bea
   \bg
  \left(q^{m+1/2};q\right)_{\infty}\sum_{n=0}^{\infty}
\frac{q^{mn+n^{2}/2}}
{(q,q^{m+1/2};q)_{n}}\qquad \qquad\qquad \\
\qquad\qquad= (-1)^mq^{2\binom{m}{2}}\left[
\frac{a_m(q^2)}{(q^2, q^8;q^{10})_\infty} - \frac{b_m(q^2)}{(q^4, q^6;q^{10})_\infty}
\right]. 
\eg
 \eea
 It is clear that the left-hand side is a $q$-analogue of the fact that 
 $I_{m+1/2}$ is a linear combination of $I_{\pm1/2}$ with coefficients 
 related to the Lommel polynomials. 
 The cases $m=0, 1$ are 
\bea 
\sum_{n=0}^{\infty}\frac{q^{n^{2}}}{(q;q)_{2n}}
&=& \frac{1}{\left(q; q^2\right)_{\infty}
(q^4, q^{16};q^{20})_\infty}, 
\label{eq2.7}\\
\sum_{n=0}^{\infty}\frac{q^{n(n+2)}}{(q;q)_{2n}}
&=& \frac{1}{\left(q; q^3\right)_{\infty}
(q^8, q^{12};q^{20})_\infty}, 
\label{eq2.8}
\eea
The identity \eqref{eq2.7} is  (79) on Lucy Slater's list, \cite{Sla}, but 
\eqref{eq2.8} is not on her list.

For $\ell\in\mathbb{N}_{0}$, we let $a=-q^{\ell+1}/z$ in (\ref{eq1.9})
and establish 
\begin{eqnarray*}
&{}& \left(z;q\right)_{\infty}{}_{1}\phi_{1}\left(-q^{\ell+1}/z;z;q,-z\right)  
=  \sum_{n=0}^{\infty}\frac{q^{2n^{2}-n}z^{2n}}{(q^{2};q^{2})_{n}}A_{q}
\left(-q^{2n+\ell}\right)\\
 & = & \frac{(-1)^{\ell}q^{-\binom{\ell}{2}}}{(q,q^{4};q^{5})_{\infty}}
 \sum_{n=0}^{\infty}\frac{a_{2n+\ell}(q)}{(q^{2};q^{2})_{n}}\left(\frac{z}
 {q^{\ell}}\right)^{2n}
 -  \frac{(-1)^{\ell}q^{-\binom{\ell}{2}}}{(q^{2},q^{3};q^{5})_{\infty}}
 \sum_{n=0}^{\infty}\frac{b_{2n+\ell}(q)}{(q^{2};q^{2})_{n}}\left(\frac{z}
 {q^{\ell}}\right)^{2n}.
\end{eqnarray*}
 This proves the following theorem. 
 \begin{thm}\label{thmgfanbn} If $\ell\in\mathbb{N}_{0}$, $z \notin 
 \{1, q^{-1}, q^{-2}, \cdots\}$ then
\begin{equation}
\begin{aligned} & \left(z;q\right)_{\infty}{}_{1}\phi_{1}\left(-q^{\ell+1}/z;z;q,-z\right)\\
 & =\frac{(-1)^{\ell}q^{-\binom{\ell}{2}}}{(q,q^{4};q^{5})_{\infty}}
 \sum_{n=0}^{\infty}\frac{a_{2n+\ell}(q)}{(q^{2};q^{2})_{n}}\left(\frac{z}
 {q^{\ell}}\right)^{2n} -\frac{(-1)^{\ell}q^{-\binom{\ell}{2}}}
 {(q^{2},q^{3};q^{5})_{\infty}}\sum_{n=0}^{\infty}\frac{b_{2n+\ell}(q)}
 {(q^{2};q^{2})_{n}}\left(\frac{z}{q^{\ell}}\right)^{2n}.
\end{aligned}
\label{eq:10}
\end{equation}
\end{thm}
 Let $z=-q^{\ell+1}$ in (\ref{eq1.9}) to get 
\begin{equation}
\left(-q^{\ell+1};q\right)_{\infty}=(-1)^{\ell}q^{-\binom{\ell}{2}}
\sum_{n=0}^{\infty} \left[\frac{a_{2n+\ell}(q)q^{2n}}
{(q,q^{4};q^{5})_{\infty}(q^{2};q^{2})_{n}}-\frac{b_{2n+\ell}(q)q^{2n}}
{(q^{2},q^{3};q^{5})_{\infty}(q^{2};q^{2})_{n}}\right].
\label{eq:11}
\end{equation}
 
From \eqref{eq:series and identities airy 1} to obtain 
\bea
\qquad A_{q}\left(-b\right)=\frac{1}{\left(q,q^{4};q^{5}\right)_{\infty}}
\sum_{k=0}^{\infty}\frac{\left(b;q\right)_{k}q^{k}a_{k}\left(q\right)}
{\left(q;q\right)_{k}}-\frac{1}{\left(q^{2},q^{3};q^{5}\right)_{\infty}}
\sum_{k=0}^{\infty}\frac{\left(b;q\right)_{k}q^{k}b_{k}\left(q\right)}
{\left(q;q\right)_{k}}. \label{eq2.9}
\eea
The special case $b = q^m$ of \eqref{eq2.9} is the interesting identity 
\begin{eqnarray}
\bg
  (-1)^mq^{\binom{m}{2}}\left[
 \frac{a_m(q)} {(q^{1},q^{4},q^{5};q^5)_\infty}
 -   \frac{b_m(q)}{(q^{2},q^{3},q^{5};q^5}_\infty \right] \\
  =   \frac{1}{(q^{1},q^{4},q^{5};q^5)_\infty}  \sum_{k=0}^{\infty}
  \frac{\left(q^{m};q\right)_{k}}{\left(q;q\right)_{k}}q^{k}a_{k}\left(q\right)
  -  \frac{1}{(q^{2},q^{3},q^{5};q^{5})_{\infty}}\sum_{k=0}^{\infty}\frac{\left(q^{m};q\right)_{k}}{(q;q)_{k}}q^{k}b_{k}(q).
  \label{eq2.10}
  \eg
\end{eqnarray}
 The further specialization, $m=1$ ,  is also interesting 
\bea
1+\sum_{k=0}^{\infty}b_{k}\left(q\right)q^{k}=
\frac{\left(q^{2},q^{3};q^{5}\right)_{\infty}}
{\left(q,q^{4};q^{5}\right)_{\infty}}
\sum_{k=0}^{\infty}a_{k}\left(q\right)q^{k}. \label{eq2.11}
\eea
The identities \eqref{eq2.10} and \eqref{eq2.11} should have very 
interesting partition theoretic interpretations.

\begin{thm}
For any 
 $m,n\in\mathbb{N}$
we have
\begin{eqnarray}
\label{eqthm2.2}
\bg
  (-1)^mq^{-\binom{m}{2}}\left[
 \frac{a_m(q)} {(q^{1},q^{4},q^{5};q^5)_\infty}
 -   \frac{b_m(q)}{(q^{2},q^{3},q^{5};q^5}_\infty \right] \\
   = (-1)^{m+n} \sum_{k=0}^{\infty}\frac{q^{k(k+m)}}{(q;q)_{k}}
 q^{-\binom{m+n+2k}{2}}\left[ \frac{a_{m+n+2k}(q)} 
 {(q ,q^{4},q^{5};q^5)_\infty}
  -   \frac{b_{m+n+2k}(q)}{(q^{2},q^{3},q^{5};q^5}_\infty \right].
\eg
\eea
\end{thm}
\begin{proof}
In \cite{Ism:Zha4} we proved that 
\begin{equation}
A_{q}\left(z\right)=\sum_{k=0}^{\infty}\frac{q^{k^{2}}}{\left(q;q\right)_{k}}
\left(-z\right)^{k}\left(w;q\right)_{k}A_{q}\left(wzq^{2k}\right).
\label{eq:series and identities airy 3}
\end{equation}
The theorem is the special case  $z = -q^m, w = q^n$ in 
\eqref{eq:series and identities airy 3}.  
\end{proof}
It must be noted that the series on the right-hand side of \eqref{eqthm2.2} 
converges because \eqref{eqmform} shows that 
\bea
\lim_{m\to \infty} \left[\frac{(-1)^m q^{-\binom{m}{2}} a_m(q)}
{(q,q^4;q^5)_\infty}
- \frac{(-1)^{m} q^{-\binom{m}{2}} b_m(q)}{(q^2,q^3;q^5)_\infty}
\right] =1. 
\notag
\eea

Using \eqref{eq1.18} and \eqref{eq2.1} we find that for $m=0, 1,
 \cdots,$ we have 
\bea
\bg
(-1)^m q^{-\binom{m}{2}}\left[\frac{ a_m(q)}
{(q,q^4;q^5)_\infty}
- \frac{ b_m(q)}{(q^2,q^3;q^5)_\infty}\right]   \qquad \qquad\qquad\\
\qquad\qquad\qquad= (-q^{m+2};q^2)_\infty 
\Sum \frac{q^{n^2+mn}}{(q^2, -q^{m+2};q^2)_n}
\eg
\label{eq2.13}
\eea
Here again the cases $m=0$ and $m=1$ are 
\bea
\frac{ 1}
{(q,q^4;q^5)_\infty} &=& (-q^{2};q^2)_\infty 
\Sum \frac{q^{n^2}}{(q^2, -q^{2};q^2)_k}, \label{eq2.14} \\
\frac{ 1}{(q^2,q^3;q^5)_\infty} &=&(-q^{3};q^2)_\infty 
\Sum \frac{q^{n^2+ n}}{(q^2, -q^{3};q^2)_n}.  \label{eq2.15} 
\eea
Neither \eqref{eq2.14} nor  \eqref{eq2.15}  seem to be on the Slater list \cite{Sla}

\section{Bilateral Sums}
 In this section we derive bilateral sum identities. 
\begin{thm} \label{thm8.1}
Let $r\ge s$, $\alpha_{j}\in\mathbb{C},\,1\le j\le r,\ \beta_{k}\in
\mathbb{C},\,1\le k\le s$, 
$|\beta_{1}/\alpha_{1}|<|w|<1$ and $|z|<|w|$ if $r=s+1$. Then 
\begin{equation}
\bg
  \sum_{m=-\infty}^{\infty}\frac{\left(\alpha_{1};q\right)_{m}}
  {\left(\beta_{1};q\right)_{m}}{}_{r}\phi_{s}\begin{pmatrix}
  \begin{array}{c}
\alpha_{1}q^{m},\dots,\alpha_{r}\\
\beta_{1}q^{m},\dots,\beta_{s}
\end{array} & \bigg|q,z\end{pmatrix}w^{m}   \qquad \qquad \qquad \\
 \qquad \qquad \qquad  =\frac{\left(q/(\alpha_{1}w),\alpha_{1}w,
 q,\beta_{1}/\alpha_{1};q\right)_{\infty}}{\left(\beta_{1},q/\alpha_{1},
 w,\beta_{1}/(\alpha_{1}w);q\right)_{\infty}}\; 
 {}_{r-1}\phi_{s-1}\left(\left.  \begin{array}{c}
  \alpha_{2},\dots,\alpha_{r}\\
\beta_{2},\dots,\beta_{s}
\end{array} \right| q , \frac{z}{w}  \right)
\eg
\label{eq3.1}
\end{equation}
\end{thm}
\begin{proof}
The left-hand side is 
\[
\begin{aligned} 
 & =\sum_{m=-\infty}^{\infty}w^{m}\sum_{n=0}^{\infty}\frac{\left(\alpha_{1};q\right)_{n+m}}{\left(\beta_{1};q\right)_{n+m}}\frac{\left(\alpha_{2},\dots,\alpha_{r};q\right)_{n}z^{n}}{\left(q,\beta_{2},\dots,\beta_{s};q\right)_{n}}\left(-q^{(n-1)/2}\right)^{n(s+1-r)}\\
 & =\sum_{n=0}^{\infty}\frac{\left(\alpha_{1},\dots,\alpha_{r};q\right)_{n}z^{n}}{\left(q,\beta_{1},\dots,\beta_{s};q\right)_{n}}\left(-q^{(n-1)/2}\right)^{n(s+1-r)}\sum_{m=-\infty}^{\infty}\frac{\left(\alpha_{1}q^{n};q\right)_{m}}{\left(\beta_{1}q^{n};q\right)_{m}}w^{m}\\
 & =\sum_{n=0}^{\infty}\frac{\left(\alpha_{1},\dots,\alpha_{r};q\right)_{n}z^{n}}{\left(q,\beta_{1},\dots,\beta_{s};q\right)_{n}}\left(-q^{(n-1)/2}\right)^{n(s+1-r)} \frac{\left(\beta_{1}/\alpha_{1},q,q^{1-n}(\alpha_{1}w),\alpha_{1}wq^{n};q\right)_{\infty}}{\left(\beta_{1}q^{n},\beta_{1}/(\alpha_{1}w),q^{1-n}/\alpha_{1},w;q\right)_{\infty}}\\
 & =\frac{\left(\beta_{1}/\alpha_{1},\alpha_{1}w,q;q\right)_{\infty}}{\left(\beta_{1},\beta_{1}/(\alpha_{1}w),w;q\right)_{\infty}}\sum_{n=0}^{\infty}\frac{\left(\alpha_{1},\dots,\alpha_{r};q\right)_{n}z^{n}}{\left(q,\alpha_{1}w,\beta_{2},\dots,\beta_{s};q\right)_{n}}\left(-q^{(n-1)/2}\right)^{n(s+1-r)}\frac{\left(q^{1-n}/(\alpha_{1}w);q\right)_{\infty}}{\left(q^{1-n}/\alpha_{1};q\right)_{\infty}}\\
 & =\frac{\left(q/(\alpha_{1}w),\alpha_{1}w,q,\beta_{1}/\alpha_{1};q\right)_{\infty}}{\left(\beta_{1},q/\alpha_{1},w,\beta_{1}/(\alpha_{1}w);q\right)_{\infty}}\sum_{n=0}^{\infty}\frac{\left(\alpha_{2},\dots,\alpha_{r};q\right)_{n}\left(z/w\right)^{n}}{\left(q,\beta_{2},\dots,\beta_{s};q\right)_{n}}\left(-q^{(n-1)/2}\right)^{n(s+1-r)}.
\end{aligned}
\]
and the theorem follows.
\end{proof}

\begin{cor} For $|\beta_{1}/\alpha_{1}|<|\alpha_{2}\alpha_{3}z/\beta_{2}|<1$
and $|\alpha_{2}\alpha_{3}|<|\beta_{2}|$ we have 
\begin{equation}
\begin{aligned} & \sum_{m=-\infty}^{\infty}\frac{\left(\alpha_{1};q\right)_{m}}{\left(\beta_{1};q\right)_{m}}{}_{3}\phi_{2}\begin{pmatrix}\begin{array}{c}
\alpha_{1}q^{m},\alpha_{2}, \alpha_3\\
\beta_{1}q^{m},\beta_{2}
\end{array} & \bigg|q,z\end{pmatrix}\left(\frac{\alpha_{2}\alpha_{3}z}{\beta_{2}}\right)^{m}\\
 & =\frac{\left(q\beta_{2}/(\alpha_{1}\alpha_{2}\alpha_{3}z),\alpha_{1}\alpha_{2}\alpha_{3}z/\beta_{2},q,\beta_{1}/\alpha_{1},\beta_{2}/\alpha_{2},\beta_{2}/\alpha_{3});q\right)_{\infty}}
 {\left(\beta_{1},\beta_{2},\beta_{2}/(\alpha_{2}\alpha_{3}),q/\alpha_{1},\alpha_{2}\alpha_{3}z/\beta_{2},\beta_{1}\beta_{2}/(\alpha_{1}\alpha_{2}\alpha_{3}z);q\right)_{\infty}}.
\end{aligned}
\label{eq3.2}
\end{equation}
\end{cor}
\begin{proof} Apply Theorem \ref{thm8.1} with $w= z \beta_2/\alpha_2\alpha_3$ and use the $q$-Gauss sum. 
\end{proof}

\begin{cor}
(i)  For $|\beta_{1}/\alpha_{1}|<|w|<1$ and $|z|<|w|$ we have 
\begin{equation}
\begin{aligned} & \sum_{m=-\infty}^{\infty}\frac{\left(\alpha_{1};q\right)_{m}}{\left(\beta_{1};q\right)_{m}}{}_{2}\phi_{1}\begin{pmatrix}\begin{array}{c}
\alpha_{1}q^{m},\alpha_{2}\\
\beta_{1}q^{m}
\end{array} & \bigg|q,z\end{pmatrix}w^{m}\\
 & =\frac{\left(q/(\alpha_{1}w),\alpha_{1}w,q,\beta_{1}/\alpha_{1},a_{2}z/w;q\right)_{\infty}}{\left(\beta_{1},q/\alpha_{1},w,\beta_{1}/(\alpha_{1}w),z/w;q\right)_{\infty}}.
\end{aligned}
\label{eq:3}
\end{equation}
(ii) For $|\beta_{1}/\alpha_{1}|<|\alpha_{2}z/\beta_{2}|<1$ we have
For $|\beta_{1}/\alpha_{1}|<|\alpha_{2}z/\beta_{2}|<1$ we
have
\begin{equation}
\begin{aligned} & \sum_{m=-\infty}^{\infty}\frac{\left(\alpha_{1};q\right)_{m}}{\left(\beta_{1};q\right)_{m}}{}_{2}\phi_{2}\begin{pmatrix}\begin{array}{c}
\alpha_{1}q^{m},\alpha_{2}\\
\beta_{1}q^{m},\beta_{2}
\end{array} & \bigg|q,z\end{pmatrix}\left(\frac{\alpha_{2}z}{\beta_{2}}\right)^{m}\\
 & =\frac{\left(q/(\alpha_{1}w),\alpha_{1}w,q,\beta_{1}/\alpha_{1},\beta_{2}/\alpha_{2};q\right)_{\infty}}{\left(\beta_{1},\beta_{2},q/\alpha_{1},w,\beta_{1}/(\alpha_{1}w);q\right)_{\infty}}.
\end{aligned}
\label{eq:4}
\end{equation}
\end{cor}
\begin{proof} The results  follow from Theorem \ref{thm8.1}   and the 
$q$-binomial theorem  and a confluent limit of the $q$-Gauss theorem. 
\end{proof}

It is important we use the following definition of $q$-shifted factorial,
\begin{equation}
\left(a;q\right)_{\infty}=\prod_{k=0}^{\infty}\left(1-aq^{k}\right), 
\left(a;q\right)_{n}=\frac{\left(a;q\right)_{\infty}}{\left(aq^{n};q\right)_{\infty}},\label{eq:1.1}
\end{equation}
where $a,q,n\in\mathbb{C}$, $|q|<1$, and 
$aq^{n}\neq q^{-k},\ k\in\mathbb{N}_{0}$,
then it is clear that
\begin{equation}
\left(a;q\right)_{n+m}=\left(a;q\right)_{n}\left(aq^{n};q\right)_{m}, 
\quad a,n,m\in\mathbb{C}.\label{eq:1.2}
\end{equation}

In the rest of this section we derive identities involving 
$I_{\nu}^{(2)}(z;q)$. 
The next theorem uses Ramanujan's ${}_1\psi_1$ sum 
\cite[(II.28)]{Gas:Rah}
 \bea
 \label{eqRam1psi1}
 \sum_{-\infty}^\infty \frac{(a;q)_n}{(b;q)_n} z^n 
 = \frac{(q, b/a, az, q/az;q)_\infty}{(b, q/a, z, b/az;q)_\infty}, \quad 
 |b/a|< |z|<1. 
\eea
\begin{thm} If 
$|q^{\nu+1}/a|< |w|<1$, then 
\bea 
\bg
\sum_{m=-\infty}^{\infty}(a;q)_{m}\left( \frac{w}{z}\right)^{m}
I_{m+\nu}^{(2)}(2z;q)  \qquad  \qquad  \qquad \\
  =\frac{z^{\nu}\left(q^{\nu+1}/a,aw,q/aw;q\right)_{\infty}}
  {\left(q/a,q^{\nu+1}/aw, w;q\right)_{\infty}}
 \sum_{n=0}^{\infty}\left(-\frac{z^2q^{\nu+1}}{aw}\right)^{n}
 q^{\binom{n}{2}}
 \frac{(w;q)_{n}}{(q,q^{\nu+1}/a;q)_{n}}\\
 = \frac{z^{\nu}\left(q^{\nu+1}/a,aw,q/aw;q\right)_{\infty}}
  {\left(q/a,q^{\nu+1}/aw, w;q\right)_{\infty}} {}_1\f_1(w;q^{\nu+1}/a;q, 
  z^2q^{\nu+1}/aw).
\eg
\label{eq:5}
\eea
\end{thm}
\begin{proof}
The definition \eqref{eqJnu2} implies that 
\[
I_{m+\nu}^{(2)}(z;q)=\frac{(z/2)^{\nu}}{(q;q)_{\infty}}\sum_{n=0}^{\infty}
\frac{q^{n^{2}+n\nu}}{(q;q)_{n}}\left(\frac{z^{2}}{4}\right)^{n}
\left(q^{n+\nu+1};q\right)_{\infty}\frac{\left(zq^{n}/2\right)^{m}}
{(q^{n+\nu+1};q)_{m}}.
\]
Assuming $|q^{\nu+1}/a| < |wz/2|<1$ we use the above expansion 
and establish the generating function 
\[
\begin{aligned} & \sum_{m=-\infty}^{\infty}(a;q)_{m}w^{m}
I_{m+\nu}^{(2)}(z;q)=\frac{(z/2)^{\nu}}{(q;q)_{\infty}}
\sum_{n=0}^{\infty}\frac{q^{n^{2}+n\nu}}{(q;q)_{n}}
\left(\frac{z^{2}}{4}\right)^{n}\left(q^{n+\nu+1};q\right)_{\infty}\\
 & \times\sum_{m=-\infty}^{\infty}
 \frac{(a;q)_{m}\left(wzq^{n}/2\right)^{m}}
 {\left(q^{n+\nu+1};q\right)_{m}}=\frac{(z/2)^{\nu}}{(q;q)_{\infty}}
 \sum_{n=0}^{\infty}\frac{q^{n^{2}+n\nu}}{(q;q)_{n}}
 \left(\frac{z^{2}}{4}\right)^{n}\left(q^{n+\nu+1};q\right)_{\infty}\\
 & \times
 \frac{\left(q,q^{n+\nu+1}/a,awzq^{n}/2,2q/(awzq^{n});q\right)_{\infty}}
 {\left(q^{n+\nu+1},q/a,wzq^{n}/2,2q^{\nu+1}/(awz);q\right)_{\infty}}\\
 & =\frac{(z/2)^{\nu}\left(q^{\nu+1}/a,awz/2,2q/(awz);q\right)_{\infty}}
 {\left(q/a,2q^{\nu+1}/(awz),wz/2;q\right)_{\infty}}
 \sum_{n=0}^{\infty}\left(- \frac{zq^{\nu+1}} {2aw}\right)^{n}
 q^{\binom{n}{2}}\frac{\left(wz/2;q\right)_{n}}
 {\left(q,q^{\nu+1}/a;q\right)_{n}},
\end{aligned}
\]
where the Ramanujan ${}_1\psi_1$ was used in the second 
to last step. This establishes our theorem. 
\end{proof}

\section{Series Involving $A_q$ and $S_n$}
 In this section we prove the identities contained in Theorems 
 \ref{Thm5.1} and \ref{Thm5.2}  and consider some of their  corollaries.  
The proofs uses the identities 
\bea
S_{n}\left(ab;q\right)&=& b^{n}\sum_{k=0}^{n}\frac{\left(b^{-1};q\right)_{k}
(-1)^kq^{-nk}q^{\binom{k+1}{2}}}{\left(q;q\right)_{k}}S_{n-k}
\left(aq^{k};q\right),
\label{eq:stieltjes 5.4}\\
S_{2n+1}\left(q^{-2n-1};q\right)&=&0,\quad S_{2n}\left(q^{-2n};q\right)=
\frac{\left(-1\right)^{n}q^{-n^{2}}}{\left(q^{2};q^{2}\right)_{n}}.
\label{eq:stieltjes 5.6}\\
S_{n}\left(-q^{-n+1/2};q\right)&=&\frac{q^{-\left(n^{2}-n\right)/4}}
{\left(q^{1/2};q^{1/2}\right)_{n}}, 
 \quad S_{n}\left(-q^{-n-1/2};q\right)=\frac{q^{-\left(n^{2}+n\right)/4}}
{\left(q^{1/2};q^{1/2}\right)_{n}}. 
\label{eq:stieltjes 5.8}
\eea
These identities were proved in Section 6 of our paper \cite{Ism:Zha3}. 

\begin{thm}\label{Thm5.1}
For $n=0,1,\dots$we have
\bea
\frac{q^{\binom{n}{2}}S_{n}\left(bq^{-n};q\right)}{\left(-b\right)^{n}}
&=&\sum_{j=0}^{\left\lfloor n/2\right\rfloor }
\frac{q^{j^{2}-j}\left(-1\right)^{j}}{\left(q^{2};q^{2}\right)_{j}}
\frac{\left(b^{-1};q\right)_{n-2j}}{\left(q;q\right)_{n-2j}},
\label{eq:stieltjes 31}\\
\frac{S_{n}\left(-q^{-n+1/2}b;q\right)q^{\binom{n}{2}}}{\left(-b\right)^{n}}
&=& \sum_{k=0}^{n}\frac{q^{\left(k^{2}-k\right)/4}}
{\left(q^{1/2};q^{1/2}\right)_{k}}\frac{\left(b^{-1};q\right)_{n-k}}
{\left(q;q\right)_{n-k}} (-1)^k,\label{eq:stieltjes 33}\\
\frac{S_{n}\left(-q^{-n-1/2}b;q\right)q^{\binom{n}{2}}}{\left(-b\right)^{n}}
&=& \sum_{k=0}^{n}\frac{q^{\left(k^{2}-3k\right)/4}}
{\left(q^{1/2};q^{1/2}\right)_{k}}\frac{\left(b^{-1};q\right)_{n-k}}
{\left(q;q\right)_{n-k}}(-1)^k,\label{eq:stieltjes 35}
\eea
\end{thm}
\begin{proof} The proof of \eqref{eq:stieltjes 31} follows from 
\eqref{eq:stieltjes 5.4} and \eqref{eq:stieltjes 5.6}.  Formulas 
\eqref{eq:stieltjes 33} and \eqref{eq:stieltjes 35} follow from 
\eqref{eq:stieltjes 5.4} and both parts of \eqref{eq:stieltjes 5.8}. 
\end{proof}

The proof of the next theorem makes use of 
\bea
I_{\nu}^{(2)}\left(2q^{-n/2};q\right)=\frac{q^{\nu n/2}
S_{n}\left(-q^{-\nu-n};q\right)}{\left(q^{n+1};q\right)_{\infty}}
=\frac{ q^{-\nu n/2}S_{n}\left(-q^{\nu-n};q\right)}
{\left(q^{n+1};q\right)_{\infty}}, 
\label{eqspecialbessel}
\eea
from our work \cite{Ism:Zha3}. 

\begin{thm}\label{Thm5.2} The Ramanujan function $A_q$ satisfies 
the following summation theorems
\bea
\sum_{k=0}^{\infty}\frac{q^{\binom{k+1}{2}}A_{q}\left(aq^{k-2n}\right)}
{\left(q;q\right)_{k}}&=& (-q;q)_\infty(aq^{1-2n};q^2)_\infty, \label{eq5.8}
\\
\sum_{k=0}^{\infty}\frac{q^{\binom{k}{2}}\left(-q^{3/2}\right)^{k}
A_{q}\left(-q^{k-n+1/2}\right)}{\left(q;q\right)_{k}}&=&
\frac{\left(q;q\right)_{n}}{q^{\left(n^{2}-n\right)/4}
\left(q^{1/2};q^{1/2}\right)_{n}},\label{eq:stieltjes 39}
\\
\sum_{k=0}^{\infty}\frac{q^{k^{2}/2}\left(-1\right)^{k}
A_{q}\left(-q^{k-n-1/2}\right)}{\left(q;q\right)_{k}}&=&\frac{\left(-1\right)^{n}
\left(q;q\right)_{n}}{q^{\left(n^{2}+n\right)4}
\left(q^{1/2};q^{1/2}\right)_{n}}.
\label{eq:stieltjes 40}
\eea
Moreover we have the expansions
\bea
A_{q}\left(w\right)&=&\left(wq;q\right)_{\infty}\sum_{n=0}^{\infty}
\frac{q^{3n^{2}}\left(-w^{2}\right)^{n}}
{\left(q^{2},wq,wq^{2};q^{2}\right)_{n}}. \label{eq:stieltjes 41}\\
A_{q}\left(w\right)&=&\left(-wq^{1/2};q\right)_{\infty}\sum_{n=0}^{\infty}
\frac{q^{n\left(3n-1\right)/4}\left(-w\right)^{n}}
{\left(q^{1/2};q^{1/2}\right)_{n}\left(-wq^{1/2};q\right)_{n}},
\label{eq:stieltjes 42}\\
A_{q}\left(w\right)&=&\left(-wq^{3/2};q\right)_{\infty}\sum_{n=0}^{\infty}
\frac{q^{n\left(3n+1\right)/4}\left(-w\right)^{n}}
{\left(q^{1/2};q^{1/2}\right)_{n}\left(-wq^{3/2};q\right)_{n}}.
\label{eq:stieltjes 43}
\eea
\end{thm}

The special cases $a =1$ and $a =q$ of \eqref{eq5.8} are worth recording. For $n >0$ they are 
\bea
\sum_{k=0}^{\infty}\frac{q^{\binom{k+1}{2}}A_{q}\left(q^{k-2n-1}\right)}
{\left(q;q\right)_{k}}&=&0,\label{eq:stieltjes 37}\\
\sum_{k=0}^{\infty}\frac{q^{\binom{k+1}{2}}A_{q}\left(q^{k-2n}\right)}{\left(q;q\right)_{k}}&=&\left(-1\right)^{n}q^{-n^{2}}\left(q;q^{2}\right)_{n},\label{eq:stieltjes 38}
\eea
 
\begin{proof}[Proof of Theorem of \ref{Thm5.2}]
The left-hand side of \eqref{eq5.8} is 
\bea
\notag
\bg
\sum_{k,m=0}^\infty \frac{(-a)^m q^{m^2+ \binom{k+1}{2}}}
{(q;q)_k(q;q)_m}  q^{km-2nm}= \sum_{m=0}^\infty
 \frac{(-a)^m q^{m^2-2mn}} {(q;q)_m} (-q^{m+1};q)_\infty\\
 = (-q;q)_\infty \sum_{m=0}^\infty
 \frac{(-a)^m q^{m^2-2mn}} {(q^2;q^2)_m}, 
\eg
\eea
and we have proved \eqref{eq5.8}.  
Let us first rewrite \eqref{eq:series and identities bessel 5 a}  in the 
equivalent form  
\bea
I_{\nu}^{(2)}\left(2z;q\right)&=&\frac{z^{\nu}}{\left(q;q\right)_{\infty}}
\sum_{k=0}^{\infty}\frac{\left(-q^{\nu}\right)^{k}}{\left(q;q\right)_{k}}
q^{\binom{k+1}{2}}A_{q}\left(-q^{\nu+k}z^{2}\right).  
\label{eqIassumofAq}
\eea
Formulas \eqref{eq:stieltjes 39} and \eqref{eq:stieltjes 40} are the 
special cases $\nu =1/2, -1/2$, respectively, combined with 
\eqref{eqspecialbessel}.  We now come to \eqref{eq:stieltjes 41}. Write  
$(wq;q)_\infty/(qw;q)_{2n}$ as $(wq^{2n+1};q)_\infty$ then expand it into 
powers of $w$ to see that the right-hand side of \eqref{eq:stieltjes 41} 
is 
\bea
\notag
\sum_{k, n=0}^\infty \frac{(-1)^{n+k} w^{2n+k}}{(q;q)_k(q^2;q^2)_n}
q^{3n^2+2nk+\binom{k+1}{2}} 
= \sum_{m=0}^\infty (-w)^m q^{\binom{m+1}{2}} 
\sum_{n=0}^{\lfloor{n/2}\rfloor} \frac{(-1)^nq^{n(n-1)}}{(q^2;q^2)_n(q;q)_{m-2n}}. 
\eea
Denote the $n$-sum by $c_m$. It is easy to see that 
\bea
\notag
\sum_{m=0}^\infty c_m t^m = \frac{(t^2;q^2)_\infty}{(t;q)_\infty} 
= (-t;q)_\infty.
\eea
Thus $c_m = q^{\binom{m}{2}}/(q;q)_m$. This establishes 
\eqref{eq:stieltjes 41}.  The proof of \eqref{eq:stieltjes 42} is similar. 
Its right-hand side is 
\bea
\notag
\bg
\sum_{n=0}^{\infty}\frac{q^{n\left(3n-1\right)/4}\left(-w\right)^{n}}
{\left(q^{1/2};q^{1/2}\right)_{n}} (-wq^{n+1/2};q)_\infty =
\sum_{n,k=0}^{\infty}\frac{q^{n\left(3n-1\right)/4} w^{n+k}}
{\left(q^{1/2};q^{1/2}\right)_{n}}  \frac{(-1)^n}{(q;q)_k}q^{nk+k^2/2}\\
= \sum_{m=0}^\infty w^m q^{m^2/2} 
\sum_{n+k=m}  \frac{(-1)^nq^{n(n-1)/4}}{(q^{1/2};q^{1/2})_n} \;
\frac{1}{(q;q)_k}. 
\eg
\eea
Denoting the inner sum by $d_m$ we find that 
\bea
\sum_{m=0}^\infty d_m t^m = \frac{(t;q^{1/2})_\infty}{(t;q)_\infty} 
= (tq^{1/2};q)_\infty,
\notag
\eea 
and we conclude that $d_m = q^{m^2/2}(-1)^m/(q;q)_m$ and the proof 
of \eqref{eq:stieltjes 42} is complete. The proof of \eqref{eq:stieltjes 43} 
is parallel to the proof of \eqref{eq:stieltjes 42} and will be omitted. 
\end{proof}
The cases $w =-q^m$ of \eqref{eq:stieltjes 41}-\eqref{eq:stieltjes 43} are 
interesting. They are 
\bea
\label{eq4.17}
\bg
\frac{(-1)^m q^{-\binom{m}{2}} a_m(q)}{(q,q^4;q^5)_\infty}
- \frac{(-1)^{m} q^{-\binom{m}{2}} b_m(q)}{(q^2,q^3;q^5)_\infty}\\
= (-q^{m+1};q)_\infty\Sum \frac{(-1)^n q^{3n^2+2mn}}
{(q^2;q^2)_n(-q^{m+1};q)_n},  
\eg
\eea
\bea
\bg  
\frac{(-1)^m q^{-m(m-1)} a_m(q^2)}{(q^2,q^8;q^{10})_\infty}
- \frac{(-1)^{m} q^{-m(m-1)} b_m(q^2)}{(q^4,q^6;q^{10})_\infty} \\
= (q^{2m+1};q)_\infty\Sum \frac{q^{2mn+ n(3n+1)/2}}
{(q;q)_n(q^{2m+1};q^2})_\infty                                                                                       \\
= (q^{2m+3};q)_\infty\Sum \frac{q^{2mn+ n(3n+1)/2}}
{(q;q)_n(q^{2m+3};q^2})_\infty.
\eg
\label{eq4.18}
\eea
The case $m=0$ of \eqref{eq4.18} is on the Slater list, see (44)--(45) in \cite{Sla}.

The structure of the identities \eqref{eq:stieltjes 42} and 
\eqref{eq:stieltjes 43} suggests that we consider the series 
\bea
\notag
\bg
(-wq^{\al+1/2};q)_\infty\Sum \frac{q^{n\left(3n+4c\right)/4}\left(-w\right)^{n}}{\left(q^{1/2};q^{1/2}\right)_{n}\left(-wq^{\al+1/2};q\right)_{n}} \\
= \sum_{n,k=0}^\infty  \frac{(-1)^n q^{n\left(3n+4c\right)/4}w^{n+k}}{\left(q^{1/2};q^{1/2}\right)_{n} (q;q)_k} q^{k(\al+n) +k^2/2}\\
=  \sum_{m=0}^\infty  w^m q^{\al m + m^2/2}
\sum_{n=0}^m \frac{(-1)^n q^{n(c-\al)+n^2/4}}
{(q^{1/2};q^{1/2})_n (q;q)_{m-n}}
\eg
\eea
Denote the $n$ sum by $c_m$. Thus 
\bea
\sum_{m=0}^\infty c_m t^m = \frac{(tq^{c-\al+1/4};q^{1/2})_\infty}
{(t;q)_\infty}.
\notag
\eea
At this stage we take $c= \al -1/4$, or $c= \al -3/4$ and find that 
$c_m = (-1)^mq^{m^2/2}/(q;q)_m$, or 
$c_m = (-1)^mq^{m(m-1)/2}/(q;q)_m$, respectively. This essentially 
leads to  \eqref{eq:stieltjes 42} and 
\eqref{eq:stieltjes 43}.   On the hand hand we have other choices for 
$c$. 

We let $c= \al -1/4 -j/2$. In this case it turned out that there is no loss 
of generality in assuming $\al=0$.  We also assume $j>0$, since 
$j=0$ is already covered. Thus 
\bea
\notag
\bg
(-wq^{1/2};q)_\infty \Sum \frac{q^{n\left(3n-2j-1\right)/4} 
\left(-w\right)^{n}}{\left(q^{1/2};q^{1/2}\right)_{n}\left(-wq^{1/2};q\right)_{n}} \\
=  \sum_{m=0}^\infty  w^m q^{m^2/2} \sum_{n=0}^m 
\frac{(-1)^n q^{-jn/2+n(n-1)/4}}{(q;q)_k(q^{1/2};q^{1/2})_n}. 
\eg
\eea
Again we denote the $n$-sum by $c_m$ and find that 
\bea
\notag
\sum_{m=0}^\infty c_m t^m = \frac{(tq^{-j/2};q^{1/2})_\infty}{(t;q)_\infty} 
= (tq^{-j/2};q^{1/2})_{j-1} (tq^{-1/2};q)_\infty.
\eea
Therefore 
\bea
\notag
c_m = (-1)^m\sum_{s+k=m} {j-1 \brack s}_{q^{1/2}} q^{s(s-1)/4} q^{-js/2}
\frac{q^{k^(k-2)/2}}{(q;q)_k}.
\eea
This proves that 
\bea
\bg
(-wq^{1/2};q)_\infty \Sum \frac{q^{n\left(3n -2j-1\right)/4} 
\left(-w\right)^{n}}{\left(q^{1/2};q^{1/2}\right)_{n}
\left(-wq^{1/2};q\right)_{n}}   \qquad     \qquad    \qquad           \\
\qquad    \qquad    = \sum_{s=0}^j {j-1 \brack s}_{q^{1/2}} 
q^{s(3s-2j-1)/4} w^s  A_q(wq^{s-1}). 
\eg
\eea
It is clear that \eqref{eq:stieltjes 43} is the special case $j=1$, with 
$w \to wq$. 

\begin{lem}
\label{lem:sw1} The Stieltjes-Wigert polynomials satisfy the generating
function 
\begin{equation}
\left(z,zt;q\right)_{\infty}=\sum_{k=0}^{\infty}\left(-z\right)^{k}q^{\binom{k}{2}}S_{k}\left(-tq^{-k};q\right).\label{eq:sw2}
\end{equation}
For $n=0,1,\dots$ we have following evaluations

\begin{equation}
S_{2n+1}\left(q^{-2n-1};q\right)=0,\quad S_{2n}\left(q^{-2n};q\right)=\frac{\left(-1\right)^{n}q^{n-n^{2}}}{\left(q^{2};q^{2}\right)_{n}}.\label{eq:sw3}
\end{equation}
and

\begin{equation}
S_{n}\left(-q^{-n+1/2};q\right)=\frac{q^{-\left(n^{2}-n\right)/4}}{\left(q^{1/2};q^{1/2}\right)_{n}},\ S_{n}\left(-q^{-n-1/2};q\right)=\frac{q^{-\left(n^{2}+n\right)/4}}{\left(q^{1/2};q^{1/2}\right)_{n}}.\label{eq:sw4}
\end{equation}
\end{lem}
\begin{proof}
From (\ref{eq:qe1}) and (\ref{eq:sw1}) we get,
\[
\begin{aligned} & \left(-zq^{1/2},zt;q\right)_{\infty}=\frac{1}{\sqrt{2\pi\log q^{-1}}}\int_{-\infty}^{\infty}\frac{\left(zt;q\right)_{\infty}\exp\left(\frac{x^{2}}{\log q^{2}}\right)}{\left(ze^{ix};q\right)_{\infty}}dx\\
& =\sum_{k=0}^{\infty}z^{k}\int_{-\infty}^{\infty}\frac{\left(te^{-ix};q\right)_{k}\exp\left(\frac{x^{2}}{\log q^{2}}+ikx\right)}{\sqrt{2\pi\log q^{-1}}\left(q;q\right)_{k}}dx=\sum_{k=0}^{\infty}z^{k}\int_{-\infty}^{\infty}\frac{\left(te^{ix};q\right)_{k}\exp\left(\frac{x^{2}}{\log q^{2}}-ikx\right)}{\sqrt{2\pi\log q^{-1}}\left(q;q\right)_{k}}dx\\
& =\sum_{k=0}^{\infty}z^{k}q^{k^{2}/2}S_{k}\left(tq^{-k-1/2};q\right).
\end{aligned}
\]
Let $t=q^{1/2}$ in (\ref{eq:sw2}) to get (\ref{eq:sw3}) by matching
the corresponding coefficients of $z^{n}$ of 
\[
\left(z^{2}q;q^{2}\right)_{\infty}=\sum_{n=0}^{\infty}\frac{q^{n^{2}}(-1)^{n}z^{2n}}{(q^{2};q^{2})_{n}}=\sum_{k=0}^{\infty}\left(-z\right)^{k}q^{\binom{k}{2}}S_{k}\left(q^{-k};q\right).
\]
Let $t=-1$ and $t=-q$ in (\ref{eq:sw2}) to obtain 
\[
\left(-zq^{1/2},-z;q\right)_{\infty}=\left(-z;q^{1/2}\right)_{\infty}=\sum_{n=0}^{\infty}\frac{q^{n(n-1)/4}z^{n}}{\left(q^{1/2};q^{1/2}\right)_{n}}=\sum_{k=0}^{\infty}z^{k}q^{k^{2}/2}S_{k}\left(-q^{-k-1/2};q\right),
\]
\[
\left(-zq^{1/2},-zq;q\right)_{\infty}=\left(-zq^{1/2};q^{1/2}\right)_{\infty}=\sum_{n=0}^{\infty}\frac{q^{n(n+1)/4}z^{n}}{\left(q^{1/2};q^{1/2}\right)_{n}}=\sum_{k=0}^{\infty}z^{k}q^{k^{2}/2}S_{k}\left(-q^{-k+1/2};q\right)
\]
respectively, then evaluations in (\ref{eq:sw4}) are obtained by
matching corresponding coefficients of $z^{n}$. \end{proof}

\begin{thm}
For $\left|c/(ab)\right|<1$ we have
\begin{equation}
\sum_{n=0}^{\infty}\frac{\left(b;q\right)_{n}q^{n^{2}/2}S_{n}\left(aq^{-n-1/2};q\right)}{\left(c;q\right)_{n}}\left(\frac{c}{ab}\right)^{n}=\frac{\left(c/b;q\right)_{\infty}}{\left(c;q\right)_{\infty}}\sum_{n=0}^{\infty}\frac{q^{n^{2}/2}}{\left(q;q\right)_{n}}\left(\frac{c}{ab}\right)^{n}A_{q}\left(\frac{cq^{n-1/2}}{a}\right).\label{eq:sw5}
\end{equation}
In particular, for \textup{$\left|cz\right|<q^{1/2}$ we have} 
\begin{equation}
\sum_{n=0}^{\infty}\frac{q^{\binom{n}{2}}}{\left(q;q\right)_{n}}\left(cz\right)^{n}A_{q}\left(cq^{n-1}\right)=\frac{\left(c;q\right)_{\infty}}{\left(cz;q\right)_{\infty}}\sum_{n=0}^{\infty}\frac{q^{n^{2}}\prod_{j=0}^{2n-1}\left(z-q^{j}\right)\left(-c^{2}\right)^{n}}{\left(q^{2},c,cq;q^{2}\right)_{n}}\label{eq:sw6}
\end{equation}
and
\end{thm}
\begin{equation}
A_{q}(c)=\sum_{n=0}^{\infty}\frac{q^{3n^{2}+n}\left(-c^{2}\right)^{n}}{\left(q^{2},cq,cq^{2};q^{2}\right)_{n}}.\label{eq:sw7}
\end{equation}
Similarly, for $\left|c^{2}z\right|<q^{-1}$ we have 
\begin{equation}
\sum_{n=0}^{\infty}\frac{q^{n^{2}+n}\left(-c^{2}z\right)^{n}}{\left(q^{2};q^{2}\right)_{n}}A_{q^{2}}\left(-c^{2}q^{2n}\right)=\frac{\left(c^{2}q;q^{2}\right)_{\infty}}{\left(c^{2}zq;q^{2}\right)_{\infty}}\sum_{n=0}^{\infty}\frac{q^{\binom{n+1}{2}}\prod_{j=0}^{n-1}\left(z-q^{2j}\right)\left(-c^{2}\right)^{n}}{\left(q,cq^{1/2},-cq^{1/2};q\right)_{n}}\label{eq:sw8}
\end{equation}
and
\begin{equation}
A_{q^{2}}\left(-c^{2}\right)=\left(c^{2}q;q^{2}\right)_{\infty}\sum_{n=0}^{\infty}\frac{q^{(3n^{2}-n)/2}c^{2n}}{\left(q,cq^{1/2},-cq^{1/2};q\right)_{n}}.\label{eq:sw9}
\end{equation}

\begin{proof}
Let $\alpha=-n$ in (\ref{eq:sw2}) to get 
\[
q^{n^{2}/2}S_{n}\left(xq^{-n-1/2};q\right)=\frac{1}{\sqrt{\pi\log q^{-2}}}\int_{-\infty}^{\infty}\frac{\left(xe^{iy};q\right)_{n}}{\left(q;q\right)_{n}}\exp\left(\frac{y^{2}}{\log q^{2}}-iny\right)dy.
\]
Under the condition $\left|c/(ab)\right|<1$ we first apply the q-Gauss
sum (II.8) of \cite{Gas:Rah} then apply the q-binomial theorem (II.1) of \cite{Gas:Rah} to obtain
\[
\begin{aligned} & \sum_{n=0}^{\infty}\frac{\left(b;q\right)_{n}q^{n^{2}/2}S_{n}\left(aq^{-n-1/2};q\right)}{\left(c;q\right)_{n}}\left(\frac{c}{ab}\right)^{n}=\int_{-\infty}^{\infty}\sum_{n=0}^{\infty}\frac{\left(b,ae^{iy};q\right)_{n}}{\left(q,c;q\right)_{n}}\left(\frac{ce^{-iy}}{ab}\right)^{n}\frac{\exp\left(\frac{y^{2}}{\log q^{2}}\right)dy}{\sqrt{\pi\log q^{-2}}}\\
& =\int_{-\infty}^{\infty}\frac{\left(ce^{-iy}/a,c/b;q\right)_{\infty}}{\left(c,ce^{-iy}/ab;q\right)_{\infty}}\frac{\exp\left(\frac{y^{2}}{\log q^{2}}\right)dy}{\sqrt{\pi\log q^{-2}}}=\int_{-\infty}^{\infty}\frac{\left(ce^{iy}/a,c/b;q\right)_{\infty}}{\left(c,ce^{iy}/ab;q\right)_{\infty}}\frac{\exp\left(\frac{y^{2}}{\log q^{2}}\right)dy}{\sqrt{\pi\log q^{-2}}}\\
& =\frac{\left(c/b;q\right)_{\infty}}{\left(c;q\right)_{\infty}}\sum_{n=0}^{\infty}\frac{1}{\left(q;q\right)_{n}}\left(\frac{c}{ab}\right)^{n}\int_{-\infty}^{\infty}\left(ce^{iy}/a;q\right)_{\infty}\frac{\exp\left(\frac{y^{2}}{\log q^{2}}+iny\right)dy}{\sqrt{\pi\log q^{-2}}}\\
& =\frac{\left(c/b;q\right)_{\infty}}{\left(c;q\right)_{\infty}}\sum_{n=0}^{\infty}\frac{q^{n^{2}/2}}{\left(q;q\right)_{n}}\left(\frac{c}{ab}\right)^{n}A_{q}\left(\frac{cq^{n-1/2}}{a}\right),
\end{aligned}
\]
which gives (\ref{eq:sw5}). 

Let $a=q^{1/2}$ in (\ref{eq:sw5}), under the condition $\left|c/b\right|<q^{1/2}$
by applying (\ref{eq:sw3}) to get
\[
\begin{aligned} & \frac{\left(c/b;q\right)_{\infty}}{\left(c;q\right)_{\infty}}\sum_{n=0}^{\infty}\frac{q^{\binom{n}{2}}}{\left(q;q\right)_{n}}\left(\frac{c}{b}\right)^{n}A_{q}\left(cq^{n-1}\right)=\sum_{n=0}^{\infty}\frac{\left(b;q\right)_{n}q^{\binom{n}{2}}S_{n}\left(q^{-n};q\right)}{\left(c;q\right)_{n}}\left(\frac{c}{b}\right)^{n}\\
& =\sum_{n=0}^{\infty}\frac{\left(b;q\right)_{2n}q^{\binom{2n}{2}}S_{2n}\left(q^{-2n};q\right)}{\left(c;q\right)_{2n}}\left(\frac{c}{b}\right)^{2n}=\sum_{n=0}^{\infty}\frac{(-1)^{n}\left(b;q\right)_{2n}q^{n^{2}}}{\left(c;q\right)_{2n}\left(q^{2};q^{2}\right)_{n}}\left(\frac{c}{b}\right)^{2n}\\
& =\sum_{n=0}^{\infty}\frac{q^{n^{2}}\prod_{j=0}^{2n-1}\left(1/b-q^{j}\right)\left(-c^{2}\right)^{n}}{\left(c;q\right)_{2n}\left(q^{2};q^{2}\right)_{n}}=\sum_{n=0}^{\infty}\frac{q^{n^{2}}\prod_{j=0}^{2n-1}\left(1/b-q^{j}\right)\left(-c^{2}\right)^{n}}{\left(q^{2},c,cq;q^{2}\right)_{n}},
\end{aligned}
\]
which gives (\ref{eq:sw6}). Let $b\to\infty$ in the above equation
or $z=0$ in (\ref{eq:sw6}) to obtain (\ref{eq:sw7}). Let $a=-1$
in (\ref{eq:sw5}), by (\ref{eq:sw4}) to obtain
\[
\begin{aligned} & \frac{\left(c/b;q\right)_{\infty}}{\left(c;q\right)_{\infty}}\sum_{n=0}^{\infty}\frac{q^{n^{2}/2}}{\left(q;q\right)_{n}}\left(-\frac{c}{b}\right)^{n}A_{q}\left(-cq^{n-1/2}\right)=\sum_{n=0}^{\infty}\frac{\left(b;q\right)_{n}q^{n^{2}/2}S_{n}\left(-q^{-n-1/2};q\right)}{\left(c;q\right)_{n}}\left(-\frac{c}{b}\right)^{n}\\
& =\sum_{n=0}^{\infty}\frac{\left(b;q\right)_{n}q^{\left(n^{2}-n\right)/4}}{\left(c;q\right)_{n}\left(q^{1/2};q^{1/2}\right)_{n}}\left(-\frac{c}{b}\right)^{n}=\sum_{n=0}^{\infty}\frac{q^{\binom{n}{2}/2}\prod_{j=0}^{n-1}\left(1/b-q^{j}\right)\left(-c\right)^{n}}{\left(q^{1/2},c^{1/2},-c^{1/2};q^{1/2}\right)_{n}},
\end{aligned}
\]
which gives (\ref{eq:sw8}), taking $b\to\infty$ in the above equation
to get (\ref{eq:sw9}). \end{proof}
\begin{cor}
For $m=0,1,\dots,$ we have
\begin{equation}
\sum_{n=0}^{\infty}\frac{q^{3n^{2}+(2m+1)n}\left(-1\right)^{n}}{\left(q^{2},-q^{m+1},-cq^{m+2};q^{2}\right)_{n}}=\frac{(-1)^{m}q^{-\binom{m}{2}}a_{m}(q)}{(q,q^{4};q^{5})_{\infty}}-\frac{(-1)^{m}q^{-\binom{m}{2}}b_{m}(q)}{(q^{2},q^{3};q^{5})_{\infty}},\label{eq:sw10}
\end{equation}
and
\begin{equation}
\sum_{n=0}^{\infty}\frac{q^{(3n^{2}-n)/2+2mn}}{\left(q,q^{m+1/2},-q^{m+1/2};q\right)_{n}}=\frac{(-1)^{m}q^{-m(m-1)}}{\left(q^{2m+1};q^{2}\right)_{\infty}}\left\{ \frac{a_{m}(q^{2})}{(q^{2},q^{8};q^{10})_{\infty}}-\frac{b_{m}(q^{2})}{(q^{4},q^{6};q^{10})_{\infty}}\right\} .\label{eq:sw11}
\end{equation}
\end{cor}
\begin{proof}
Equation (\ref{eq:sw10}) is obtained by letting $c=-q^{m},\ m=0,1,\dots$
in (\ref{eq:sw7}), while Equation (\ref{eq:sw11})is followed by
setting $c=q^{m},\ m=0,1,\dots$ in (\ref{eq:sw9}). \end{proof}

 \end{document}